\documentclass[10pt]{amsart}

\usepackage{amsmath,amsfonts, latexsym,graphicx, footmisc, fancyhdr}

 \usepackage{hyperref}

\hypersetup{colorlinks=true, urlcolor=blue, citecolor=blue, linkcolor=blue}

\newcommand{\U}{{\mathcal U}}
\newcommand{\0}{{\mathbf 0}}
\newcommand{\C}{{\mathbb C}}
\newcommand{\Z}{{\mathbb Z}}

\newcommand{\rank}{\mathop{\rm rank}\nolimits}
\newcommand{\arrow}[1]{\stackrel{#1}{\longrightarrow}}

\newtheorem{defn0}{Definition}[section]
\newtheorem{prop0}[defn0]{Proposition}
\newtheorem{conj0}[defn0]{Conjecture}
\newtheorem{thm0}[defn0]{Theorem}
\newtheorem{lem0}[defn0]{Lemma}
\newtheorem{corollary0}[defn0]{Corollary}
\newtheorem{example0}[defn0]{Example}
\newtheorem{remark0}[defn0]{Remark}
\newtheorem{question0}[defn0]{Question}

\newenvironment{defn}{\begin{defn0}}{\end{defn0}}
\newenvironment{prop}{\begin{prop0}}{\end{prop0}}
\newenvironment{conj}{\begin{conj0}}{\end{conj0}}
\newenvironment{thm}{\begin{thm0}}{\end{thm0}}

\newenvironment{exm}{\begin{example0}\rm}{\end{example0}}
\newenvironment{rem}{\begin{remark0}\rm}{\end{remark0}}
\newenvironment{ques}{\begin{question0}\rm}{\end{question0}}

\newcommand{\defref}[1]{Definition~\ref{#1}}
\newcommand{\propref}[1]{Proposition~\ref{#1}}
\newcommand{\thmref}[1]{Theorem~\ref{#1}}

\newcommand{\exref}[1]{Example~\ref{#1}}

\newcommand{\conjref}[1]{Conjecture~\ref{#1}}

\newcommand{\mbf}[1]{{\mathbf #1}}

\title{A New Conjecture, a New Invariant, and a New Non-splitting Result}

\author{David B. Massey}



\date{}

\begin{document}

\begin{abstract} We prove a new non-splitting result for the cohomology of the Milnor fiber, reminiscent of the classical result proved independently by Lazzeri, Gabrielov, and L\^e in 1973-74.

	We do this while exploring a conjecture of Bobadilla about a stronger version of our non-splitting result. To explore this conjecture, we define a new numerical invariant for hypersurfaces with $1$-dimensional critical loci: the beta invariant. The beta invariant is an invariant of the ambient topological-type of the hypersurface, is non-negative, and is algebraically calculable. Results about the beta invariant remove the topology from Bobadilla's conjecture and turn it into a purely algebraic question.
\end{abstract}

\maketitle

\thispagestyle{fancy}

\lhead{}
\chead{}
\rhead{ }

\lfoot{}
\cfoot{}
\rfoot{}

\section{Introduction} 

Throughout this paper, we suppose that $\U$ is an open neighborhood of the origin in $\C^{n+1}$, and that \hbox{$f:(\U, \0)\rightarrow (\C, 0)$} is a complex analytic function with a $1$-dimensional critical locus at the origin, i.e., $\dim_\0\Sigma f=1$. We use coordinates $(z_0, \dots, z_n)$ on $\U$ and, to omit the non-reduced curve case, we assume that $n\geq 2$, which implies that $f$ is reduced.

We assume that $L$ is a linear form which is generic enough so that $\dim_\0\Sigma\big(f_{|_{V(L)}}\big)=0$. For convenience, possibly after a linear change of coordinates, we may assume that $L$ is the first coordinate $z_0$, so that we have $\dim_\0\Sigma\big(f_{|_{V(z_0)}}\big)=0$.

\medskip

We assume that $\U$ is chosen (e.g., as a small enough open ball) so that, for each irreducible component $C$ of $\Sigma f$:
\begin{enumerate}

\item[$\bullet$] $C$ contains the origin;

\item[$\bullet$] $C$ is contained in the vanishing locus $V(f)$ of $f$; and

\item[$\bullet$]  $C-\{\0\}$ is homeomorphic to a punctured disk.

\end{enumerate}

\smallskip

Furthermore, we assume that $\U$ is so small that, for each irreducible component $C$ of $\Sigma f$:

\begin{enumerate}

\item[$\bullet$] the isomorphism-type of the reduced integral cohomology groups $\widetilde H^*(F_{f,\mathbf p};\Z)$ is independent of the choice of $\mathbf p\in C-\{\0\}$. This is the same isomorphism-type as the reduced cohomology at $\mathbf p$ of the Milnor fiber of the hyperplane slice $f_{|_{V(z_0-z_0(\mathbf p))}}$, for $\mathbf p\in C-\{\0\}$ close enough to $\0$. Such a slice yields an isolated critical point at $\mathbf p$, and so this cohomology is non-zero in a single degree, namely degree $(n-1)$, and 
$$
\widetilde H^{n-1}(F_{f,\mathbf p};\Z) \ \cong \ \Z^{{\stackrel{\circ}{\mu}}_C},
$$
where ${\stackrel{\circ}{\mu}}_C$ is the Milnor number at $\mathbf p$ of $f_{|_{V(z_0-z_0(\mathbf p))}}$.

\end{enumerate}

\medskip

We can now state the classic non-splitting result, proved independently by L\^e in \cite{leacampo}, Lazzeri in \cite{lazzerimono}, and Gabrielov in \cite{gabrielov} in 1973-1974.

\begin{thm}\label{thm:nonsplit} \textnormal{(L\^e-Lazzeri-Gabrielov)} Suppose that the Milnor number of $f_{|_{V(z_0)}}$ at the origin is equal to 
$$\sum_C \left(C\cdot V(z_0)\right)_\0{\stackrel{\circ}{\mu}}_C,$$ where the sum is over the irreducible components $C$ of $\Sigma f$ and $\left(C\cdot V(z_0)\right)_\0$ is the intersection number (which would be the multiplicity of $C$ at $\0$ if $z_0$ were generic enough). That is, suppose that the Milnor number in the $z_0=0$ slice ``splits'' over the critical points of $f$ in the slice where $z_0=t$ for a small value of $t\neq0$.

Then, in fact, $\Sigma f$ has a single irreducible component which is smooth and is transversely intersected by $V(z_0)$ at $\0$.
\end{thm}

\medskip

\begin{rem} We have stated the above theorem in a slightly more general form than the original statements, but the proofs remain the same.

We should also comment that there is a pleasant priority ``dispute'' as to which of L\^e, Lazzeri, or Gabrielov first proved the above result. Many years ago, we contacted all three authors, and each one claimed that one of the other two proved the result first. 
\end{rem}

\medskip

Now, roughly 40 years after \thmref{thm:nonsplit} was proved, Javier Fern\'andez de Bobadilla has made a conjecture, which looks like it should be related to \thmref{thm:nonsplit}:

\begin{conj}\label{conj:bob} \textnormal{(Fern\'andez de Bobadilla)} Suppose that the critical locus of $f$ has a single irreducible component $C$ and that the isomorphism-type of the cohomology groups $\widetilde H^*(F_{f,\mathbf p};\Z)$ is independent of the choice of $\mathbf p\in C$, i.e., suppose that $\widetilde H^*(F_{f, \0};\Z)$ is non-zero only in degree $(n-1)$ and
$$
\widetilde H^{n-1}(F_{f, \0};\Z) \ \cong \ \Z^{{\stackrel{\circ}{\mu}}_C}.
$$ 
Then, $C$ is smooth.
\end{conj}

\medskip

However, this conjecture does {\bf not} follow from \thmref{thm:nonsplit} (in any way that has yet been tried) and the conjecture remains a conjecture.

\medskip

Before we continue, we give the obvious {\bf generalized Bobadilla conjecture} for the case where $\Sigma f$ may have more than a single irreducible component:

\begin{conj}\label{conj:genbob} Suppose that $\widetilde H^*(F_{f, \0};\Z)$ is non-zero only in degree $(n-1)$ and
$$
\widetilde H^{n-1}(F_{f, \0};\Z) \ \cong \ \bigoplus_C\Z^{{\stackrel{\circ}{\mu}}_C},
$$ 
where the sum is over all irreducible components $C$ of $\Sigma f$. Then, $\Sigma f$ has a single irreducible component, which is smooth.

\end{conj}

\bigskip

This paper represents our initial attack on the problem. Here, after recalling earlier definitions and results, we obtain the following new results:

\medskip

\begin{enumerate}

\item In \defref{def:beta}, we define a new invariant $\beta_f=\beta_{f, z_0}$, which algebraically calculable.

\smallskip

\item In \thmref{thm:betainv}, we prove that $\beta_f$ is an invariant of the the ambient topological-type of the hypersuface and, in particular, is independent of the linear form $z_0$.

\smallskip

\item In \thmref{thm:betti}, we show that $\beta_f\geq0$, and that $\beta_f=0$ implies that $\widetilde H^{n}(F_{f, \0};\Z)=0$.
\smallskip

\item We prove in \thmref{thm:betaequiv}  that, in fact, $\beta_f=0$ is precisely equivalent to the hypotheses of our generalized Bobadilla conjecture, \conjref{conj:genbob}. We thus remove the topology from the hypotheses of the conjecture. 

Furthermore, we prove in this theorem that $\beta_f=0$ implies that $\Sigma f$ has a single irreducible component, i.e., the cohomology does not ``split'' over various components. Hence, we are back in the setting of Bobadilla's original conjecture.

\smallskip

\item In \thmref{thm:beta2}, we discuss the case where $\beta_f=1$ and show that, in this case, the critical locus must have precisely two irreducible components.

\end{enumerate}

\bigskip

We thank Javier Fern\'andez de Bobadilla for discussing his conjecture with us, and L\^e D\~ung Tr\'ang for valuable conversations on this topic.

\medskip

\section{Notation and Known Results} 

Our assumption that 
$$\dim_\0\Sigma\big(f_{|_{V(z_0)}}\big) \ = \ \dim_\0V\left(z_0, \frac{\partial f}{\partial z_1}, \dots, \frac{\partial f}{\partial z_n}\right) \ = \ 0$$ is precisely equivalent to saying that $V\left(\frac{\partial f}{\partial z_1}, \dots, \frac{\partial f}{\partial z_n}\right)$ is purely $1$-dimensional (and not empty) at the origin and is properly intersected at the origin by $V(z_0)$.

In terms of analytic cycles, 
$$
V\left(\frac{\partial f}{\partial z_1}, \dots, \frac{\partial f}{\partial z_n}\right) \ = \ \Gamma^1_{f,z_0} \ + \ \Lambda^1_{f, z_0},
$$
where $\Gamma^1_{f,z_0}$ is the relative polar curve of $f$, which consists of components not contained in $\Sigma f$, and $\Lambda^1_{f, z_0}$ is the $1$-dimensional L\^e cycle, which consists of components which are contained in $\Sigma f$. See Definition 1.11 of \cite{lecycles}.

Note that $V\left(\frac{\partial f}{\partial z_0}\right)$ necessarily intersects $\Gamma^1_{f,z_0}$ properly at $\0$, and that $V(z_0)$ intersects $\Lambda^1_{f, z_0}$ properly at $\0$ by our assumption.

Letting $C$'s denote the underlying reduced components of $\Sigma f$ at $\0$, at the origin, we have
$$
\Lambda^1_{f, z_0}  \ = \ \sum_C{\stackrel{\circ}{\mu}}_C[C],
$$
where we use the square brackets to indicate that we are considering $C$ as a cycle, and ${\stackrel{\circ}{\mu}}_C$ is the Milnor number of $f$, restricted to a generic hyperplane slice, at a point $\mbf p$ on $C-\{\0\}$ close to $\0$. See Remark 1.19  of \cite{lecycles}.

The intersection numbers $\left(\Gamma^1_{f, z_0} \cdot V\left(\frac{\partial f}{\partial z_0}\right)\right)_\0$ and $\left(\Lambda^1_{f, z_0} \cdot V(z_0)\right)_\0$ are the L\^e numbers $\lambda^0_{f,z_0}$ and $\lambda^1_{f,z_0}$ (at the origin). See Definition 1.11 of \cite{lecycles}. Note that
$$\lambda^1_{f,z_0}=\sum_C \left(C\cdot V(z_0)\right)_\0{\stackrel{\circ}{\mu}}_C.$$

\bigskip

We give here a list of the numbers, other than the beta invariant, which will be used throughout this paper:

\begin{enumerate}

\item As we have used several times already, ${\stackrel{\circ}{\mu}}_C$ is the Milnor number of a generic hyperplane slice at a point $\mbf p\neq \0$ on the irreducible component $C$ of $\Sigma f$.

\medskip

\item We use the L\^e numbers $\lambda^0_{f,z_0}$ and $\lambda^1_{f,z_0}$.

\medskip

\item Throughout, we will use the Betti numbers $\tilde b_{n-1}$ and $\tilde b_n$ of the reduced integral cohomology of the Milnor fiber $F_{f, \0}$ of $f$ at the origin in degrees $(n-1)$ and $n$. (We do not need to write ``reduced'' here and, yet, we do so because we are thinking of the vanishing cycles, not the nearby cycles.)

\medskip

\item We let $\sigma_f:=\sum_C {\stackrel{\circ}{\mu}}_C$.

\smallskip

Note the lack of the intersection multiplicities in this summation. Thus, $\lambda^1_{f,z_0}= \sigma_f$ if and only if each irreducible component $C$ of $\Sigma f$ is smooth and transversely intersected at the origin by $V(z_0)$.

\medskip

\item Since we are using cohomology, not homology, $\widetilde H^{n-1}(F_{f, \0};\Z)$ is free Abelian, but $\widetilde H^{n}(F_{f, \0};\Z)$ may contain torsion. For each prime $p$, we let $\tau_p$ denote the number of $p$-torsion direct summands of $\widetilde H^{n}(F_{f, \0};\Z)$. With this notation, and our notation for the Betti numbers, the Universal Coefficient Theorem tells us that
$$
\operatorname{dim} \widetilde H^{n-1}(F_{f,\0};\Z/p\Z)= \tilde b_{n-1}+\tau_p\hskip 0.2in\textnormal{ and } \hskip 0.2in\operatorname{dim} \widetilde H^{n}(F_{f,\0};\Z/p\Z)= \tilde b_{n}+\tau_p.
$$

\medskip

\item Finally, we let $c_f$ denote the number of irreducible components of $\Sigma f$.

\end{enumerate}

\bigskip

In Corollary 10.10 of \cite{lecycles} (though there is an indexing typographical error), we proved a fundamental result linking the L\^e numbers and the Betti numbers of the Milnor fiber, which continues to hold with coefficients in $\Z/p\Z$.

\begin{thm}\label{thm:le} There is an exact sequence 
 $$0\rightarrow\widetilde H^{n-1}(F_{f,\0};\Z)\rightarrow\Z^{\lambda^1_{f, z_0}}\arrow{\delta}\Z^{\lambda^0_{f, z_0}}\rightarrow \widetilde H^{n}(F_{f,\0};\Z)\rightarrow 0.\eqno{(\dagger)},$$
 and so, 
 $$\tilde b_{n-1}\leq \lambda^1_{f, z_0},  \ \ \tilde b_{n}\leq \lambda^0_{f, z_0}, \ \textnormal{ and } \tilde b_{n}-\tilde b_{n-1} \ = \ \lambda^0_{f, z_0}-\lambda^1_{f, z_0}.$$
 
 \bigskip
 
 In addition, for each prime number $p$, there is an exact sequence 
 $$0\rightarrow\widetilde H^{n-1}(F_{f,\0};\Z/p\Z)\rightarrow(\Z/p\Z)^{\lambda^1_{f, z_0}}\arrow{\delta_p}(\Z/p\Z)^{\lambda^0_{f, z_0}}\rightarrow \widetilde H^{n}(F_{f,\0};\Z/p\Z)\rightarrow 0,\eqno{(\ddagger)}$$
 and so, 
 $$\tilde b_{n-1}+\tau_p\leq \lambda^1_{f, z_0} \hskip .4in \textnormal{ and } \hskip 0.4in  \tilde b_{n}+\tau_p\leq \lambda^0_{f, z_0}.$$ 
 \end{thm}
 
\medskip

\section{Definition of the beta invariant and examples}

\begin{defn}\label{def:beta} We define the {\bf beta invariant}:
$$
\beta_f \ = \ \beta_{f, z_0} \ := \ \left(\Gamma^1_{f, z_0} \cdot V\left(\frac{\partial f}{\partial z_0}\right)\right)_\0 - \sum_C{\stackrel{\circ}{\mu}}_C\left[\big(C\cdot V(z_0)\big)_\0-1\right] \ =$$

$$\lambda^0_{f, z_0}-\lambda^1_{f, z_0}+\sigma_f \ = \  =  \ \tilde b_{n}-\tilde b_{n-1} + \sigma_f.$$
\end{defn}

\medskip

\begin{rem} Note that the final expression above does not depend on $z_0$. Thus, the value of $\beta_{f,z_0}$ is independent of the linear form $z_0$ (provided that $f$, restricted to where the linear form is zero, has an isolated critical point). Consequently, we may drop the $z_0$ from the notation, but sometimes include it to indicate what linear form will actually be used in the calculation of $\beta_f$.
\end{rem}

\medskip

\begin{exm}\label{exm:smoothcomp} Suppose that all of the components $C$ of $\Sigma f$ are smooth and transversely intersected by $V(z_0)$ at $\0$. Then,
$$
\beta_{f, z_0} \ = \ \left(\Gamma^1_{f, z_0} \cdot V\left(\frac{\partial f}{\partial z_0}\right)\right)_\0  \ = \ \lambda^0_{f,z_0}.
$$
\end{exm}

\bigskip

Thus, the only time that $\beta_{f, z_0}$ is really a ``new'' invariant is when the critical locus itself has a singular component.

\begin{exm}
Suppose $f=z^2+(y^2-x^3)^d$, where $d\geq 2$. Both $f_{|_{V(x)}}$ and $f_{|_{V(y)}}$ have isolated critical points at the origin. We will calculate both $\beta_{f,x}$ and $\beta_{f,y}$, and see that they are the same.

First, we find that, as sets,
$$
\Sigma f  \ = \ V\left(\frac{\partial f}{\partial x}, \frac{\partial f}{\partial y}, \frac{\partial f}{\partial z}\right) \ = $$
$$
V\big(d(y^2-x^3)^{d-1}(-3x^2), \  d(y^2-x^3)^{d-1}2y, \ 2z\big) \ = \ V(y^2-x^3, z).
$$

Now, as cycles, we calculate 
$$
V\left(\frac{\partial f}{\partial y}, \frac{\partial f}{\partial z}\right) \ = \ V\big(d(y^2-x^3)^{d-1}2y,  \ 2z\big) \ =
$$
$$
V(y,z)+ (d-1)V(y^2-x^3, z)  \ = \ \Gamma^1_{f, x} + \Lambda^1_{f, x}.
$$
Thus, we have $\Gamma^1_{f, x}=V(y,z)$, and that $\Sigma f$ consists of the single component $C=V(y^2-x^3, z)$, with ${\stackrel{\circ}{\mu}}_C=d-1$. Therefore,
$$
\beta_{f,x} \ = \  \left(\Gamma^1_{f, x} \cdot V\left(\frac{\partial f}{\partial x}\right)\right)_\0 - \sum_C{\stackrel{\circ}{\mu}}_C\left[\big(C\cdot V(x)\big)_\0-1\right] \ = 
$$
\smallskip
$$\left(V(y,z)\cdot V(d(y^2-x^3)^{d-1}(-3x^2))\right)_\0 \ -(d-1)(V(y^2-x^3, z)\cdot V(x))_\0 + (d-1) \ = \ 
$$
\smallskip
$$
\left(V(y,z)\cdot V((y^2-x^3)^{d-1})\right)_\0 \ + \ \left(V(y,z)\cdot V(x^2))\right)_\0 \ -2(d-1)+(d-1) \ =
$$
$$
3(d-1)+2-(d-1) \ = \ 2d.
$$

\bigskip

To calculate $\beta_{f,y}$, we proceed similarly.

\smallskip

As cycles, we calculate 
$$
V\left(\frac{\partial f}{\partial x}, \frac{\partial f}{\partial z}\right) \ = \ V\big(d(y^2-x^3)^{d-1}(-3x^2),  \ 2z\big) \ = 
$$
$$
2V(x,z)+ (d-1)V(y^2-x^3, z)  \ = \ \Gamma^1_{f, y} + \Lambda^1_{f, y}.
$$

\noindent Thus, we have $\Gamma^1_{f, y}=2V(x,z)$, and, of course, that $\Sigma f$ consists of the single component $C=V(y^2-x^3, z)$, with ${\stackrel{\circ}{\mu}}_C=d-1$. Therefore,
$$
\beta_{f,y} \ = \  \left(\Gamma^1_{f, y} \cdot V\left(\frac{\partial f}{\partial y}\right)\right)_\0 - \sum_C{\stackrel{\circ}{\mu}}_C\left[\big(C\cdot V(y)\big)_\0-1\right] \ = 
$$
\smallskip
$$\left(2V(x,z)\cdot V(d(y^2-x^3)^{d-1}(2y))\right)_\0 \ -(d-1)(V(y^2-x^3, z)\cdot V(y))_\0 + (d-1) \ = \ 
$$
\smallskip
$$
\left(2V(x,z)\cdot V((y^2-x^3)^{d-1})\right)_\0 \ + \ \left(2V(x,z)\cdot V(y))\right)_\0 \ -3(d-1)+(d-1) \ = 
$$
$$
4(d-1)+2-2(d-1) \ = \ 2d.
$$

\medskip

As promised, we see that $\beta_{f,x}=\beta_{f,y}$, even though the separate terms in the calculation are different.

\end{exm}

\medskip

\section{Invariance} 

In this short section, we prove the topological invariance of $\beta_f$ and $\sigma_f$. 

\begin{thm}\label{thm:betainv} If  \hbox{$f:(\U, \0)\rightarrow (\C, 0)$}  and  \hbox{$g:(\U, \0)\rightarrow (\C, 0)$} are reduced with $1$-dimensional critical loci at the origin, and $V(f)$ and $V(g)$ have the same local ambient topological-type at $\0$, then $\sigma_f=\sigma_g$ and $\beta_f=\beta_g$.
\end{thm}
\begin{proof} As L\^e proved in \cite{leattach} and \cite{topsing}, the homotopy-type of the Milnor fiber is an invariant of the ambient topological-type for reduced functions $f$; thus, the topological invariance of $\beta_f $ would follow from the topological invariance of $\sum_C{\stackrel{\circ}{\mu}}_C$. However, this latter topological invariance is easy to establish. 

The singular set of $V(f)$ must map to the singular set under an ambient homeomorphism and, as we require the origin to map to the origin, the punctured singular set $\Sigma V(f)-\{\0\}$ must map to the punctured singular set, and so the components of $\Sigma f$ at the origin must map bijectively to the components of the singular set at the origin. Now the homotopy-type of the Milnor fiber of $f$ at a point $\mbf p\in \Sigma V(f)$ near $\0$ is invariant under an ambient homeomorphism, and this homotopy-type is that of a bouquet of ${\stackrel{\circ}{\mu}}_C$ $(n-1)$-spheres, where $C$ is the component of $\Sigma V(f)$ containing $\mbf p$.
\end{proof}

\medskip

\section{Non-negativity and Milnor fiber consequences}

In this section, we first need to review a number of known results, and establish some notation.

\bigskip

Recall that our choice of $\U$ implies that, for each irreducible component $C$ of $\Sigma f$, $C-\{\0\}$ is topologically a punctured disk and so, is homotopy-equivalent to a circle. There is an ``internal'' (also known as ``vertical'') monodromy action, $h_C$, on $\Z^{{\stackrel{\circ}{\mu}}_C}$ given by traveling once around this circle. This is the monodromy of the local system obtained by considering the complex of sheaves of vanishing cycles along $f$, and restricting to $C-\{\0\}$.

\smallskip

Now, a result of Siersma in \cite{siersmavarlad}, or an easy exercise using perverse sheaves (see the remark at the end of \cite{siersmavarlad}) tells us that
\begin{thm}\label{thm:incl}\textnormal{(Siersma)}  There is an inclusion
$$
\widetilde H^{n-1}(F_{f,\0};\Z)\hookrightarrow \bigoplus_{C}\operatorname{ker}\{\operatorname{id}-h_C\},
$$
which commutes with the monodromy action on the vanishing cycles along $f$.

In particular,
$$\rank\widetilde H^{n-1}(F_{f,\0};\Z)\leq \sigma_f,$$
and equality implies that each $h_C$ is the identity.

\medskip

Furthermore,  this result holds with $\Z/p\Z$ coefficients, where the proof remains identical. Thus,
$$\dim \widetilde H^{n-1}(F_{f,\0};\Z/p\Z)\leq \sigma_f.$$
\end{thm}

\bigskip 

Thus, we conclude immediately that:

\begin{thm}\label{thm:betti} For all primes $p$,
$$\tilde b_{n-1}+\,\tau_p \ \leq \ \sigma_f \ \textnormal{ and } \ \tilde b_{n}+\,\tau_p \ \leq \ \beta_f.$$
In particular, $0\leq \tilde b_{n}\leq \beta_f$.
\end{thm}

\bigskip

Recall that $c_f$ denotes the number of irreducible components of $\Sigma f$.

\begin{prop}\label{prop:trace} If $\displaystyle\rank\widetilde H^{n-1}(F_{f,\0};\Z) = \sigma_f$,
then the trace of the Milnor monodromy of $f$ on $\widetilde H^{n-1}(F_{f,\0};\Z)$ is $(-1)^nc_f$.
\end{prop}

\begin{proof} Under the assumption, \thmref{thm:incl} tells us that the trace of the Milnor monodromy on $\widetilde H^{n-1}(F_{f,\0};\Z)$ is the sum of the traces of the Milnor monodromy on each $\operatorname{ker}\{\operatorname{id}-h_C\}$. As $h_C$ is the identity, this is simply the Milnor monodromy of $f$ at a point $\mbf p$ on $C-\{\0\}$ near the origin. By A'Campo's result in \cite{acamp}, this is $(-1)^n$.
\end{proof}
\bigskip

\bigskip

The case where $\beta_f=0$ is extremely restrictive.

\begin{thm}\label{thm:betaequiv} The following are equivalent:
\begin{enumerate}
\item $\beta_f=0$; and

\smallskip

\item $\widetilde H^{n}(F_{f,\0};\Z)=0$, and $\widetilde H^{n-1}(F_{f,\0};\Z) \cong\Z^{\sigma_f}$.
\end{enumerate}

\smallskip

In addition, when these equivalent conditions hold, $\Sigma f$ has a single irreducible component $C$ and the internal monodromy automorphism $h_C$ is the identity.
\end{thm}

\begin{proof} That (2) implies (1) follows immediately from \thmref{thm:le}.  That (1) implies (2) follows immediately from \thmref{thm:betti}. That $c_f=1$ follows from \propref{prop:trace} and A'Campo's Theorem applied to the cohomology of $F_{f,\0}$. That the internal monodromy automorphism $h_C$ is the identity is part of \thmref{thm:incl}.
 \end{proof}
 
 \medskip
 
 \begin{rem} The statement in \thmref{thm:betaequiv} that $\Sigma f$ must have a single irreducible component at the origin is a ``non-splitting result'', of the flavor of the result proved independently by \cite{leacampo}, Lazzeri in \cite{lazzerimono}, and Gabrielov in \cite{gabrielov}; however, those three works use the cohomology of Milnor fibers of $f$ restricted to hyperplane slices, rather than looking at the cohomology of the Milnor fiber of $f$ itself.
 \end{rem}
 
 \bigskip
 
 The case where $\beta_f=1$ is also interesting to consider:
 
 \begin{thm}\label{thm:beta2} Suppose that $\beta_f=1$. Then, $c_f=2$, and either
 \begin{enumerate}
 \item $\tilde b_{n}=0$, $\tilde b_{n-1}=\sigma_f-1$ and, for all primes $p$, $\widetilde H^{n}(F_{f,\0};\Z)$ has at most one direct summand with $p$-torsion; or
 
 \smallskip
 
 \item $\tilde b_{n}=1$, $\tilde b_{n-1}=\sigma_f$, and $\widetilde H^{n}(F_{f,\0};\Z)$ has no torsion (and so is isomorphic to $\Z$).
 \end{enumerate}
 \end{thm}
 \begin{proof} By \thmref{thm:betti}, $\tilde b_{n}\leq \beta_f$, and we know that $\tilde b_{n}-\tilde b_{n-1}=\beta_f-\sigma_f=1-\sigma_f$. So we obtain the two cases to consider: (1) where $\tilde b_{n}=0$ and $\tilde b_{n-1}=\sigma_f-1$ and (2) where $\tilde b_{n}=1$ and $\tilde b_{n-1}=\sigma_f$. The conclusions about torsion in both cases follow from the $p$-torsion statement in \thmref{thm:betti}. All that remains for us to show is the claim about $c_f$ .
 
In case (2), by A'Campo's result, the Lefschetz number of the Milnor monodromy on $H^*(F_{f,\0};\Z)$ is zero. By \propref{prop:trace}, the trace of the monodromy of $f$ on $\widetilde H^{n-1}(F_{f,\0};\Z)$ is $(-1)^nc_f$. Since $\widetilde H^{n}(F_{f,\0};\Z)\cong \Z$, the trace of the monodromy of $f$  in degree $n$ is $\pm1$. Thus, we obtain that $1-c_f\pm1=0$. Hence, $c_f=0$ or $2$, but we are assuming that $f$ has a $1$-dimensional critical locus, so $c_f\neq 0$.

Case (1) is very similar. Since $\tilde b_{n}=0$, A'Campo's result tells us that  the trace of the monodromy of $f$ on $\widetilde H^{n-1}(F_{f,\0};\Z)$ is $(-1)^n$. On the other hand, restriction induces the inclusion 
$$
\widetilde H^{n-1}(F_{f,\0};\Z)\hookrightarrow \bigoplus_C \widetilde H^{n-1}(F_{f,\mbf p_C};\Z)\cong\Z^{\sigma_f},
$$
where $\mbf p_C$ denotes a point of $C-\{\0\}$ close to the origin. This inclusion is compatible with the $f$ monodromy action and, since $\tilde b_{n-1}=\sigma_f-1$, the cokernel is isomorphic to $\Z\oplus T$, where $T$ is pure torsion. The trace of the map induced by the $f$ monodromy on this cokernel is $\pm1$. Thus, from additivity of the traces, we obtain $(-1)^n = (-1)^nc_f\pm1$, and conclude once again that $c_f=2$.
 \end{proof}

\bigskip

As we saw in \exref{exm:smoothcomp}, if all of the components of $\Sigma f$ are smooth at $\0$, then, for generic $z_0$, $\beta_f=\lambda_{f,z_0}^0$; in this case, results on $\lambda_{f,z_0}^0$ imply even stronger results when $\beta_f=0$ or $1$.

For instance, the non-splitting result of L\^e -Lazzeri-Gabrielov immediately implies the first item below, while the main theorem of \cite{lemassey} immediately implies the second item.

\begin{prop} Suppose that all of the components of $\Sigma f$ are smooth at $\0$.

\begin{enumerate}
\item If $\beta_f=0$, then $\Sigma f$ has a unique (smooth) component $C$ at the origin, along which the Milnor number of a generic hyperplane slice is constant. In particular, $\widetilde H^n(F_{f,\0})=0$ and $\widetilde H^{n-1}(F_{f,\0})\cong \Z^{{\stackrel{\circ}{\mu}}_C}$
\item If $\beta_f=1$, then $\widetilde H^n(F_{f,\0})=0$ and $\widetilde H^{n-1}(F_{f,\0})\cong \Z^{\sigma_f -1}$.
\end{enumerate}
\end{prop}

\medskip

\section{Concluding Remarks}

As we stated earlier, our interest in the beta invariant is when the critical locus of $f$ is itself singular, since the reason we defined the beta invariant is because it arose naturally while we were considering the conjecture, \conjref{conj:bob}, of Fern\'andez de Bobadilla. We now refer to this conjecture as the {\bf beta conjecture} or the {\bf $\beta_f$ conjecture}.

\smallskip

The beta conjecture is related to another conjecture: L\^e's conjecture (see, for instance, \cite{bobleconj}). The suspicion is that the proof of the beta conjecture will be very difficult, and will require new techniques. Good candidates for counterexamples are also hard to produce. 

We believe that viewing the problem in terms of $\beta_f$ may help regardless of whether the conjecture is true or false. If the beta conjecture is true, describing the question in terms of $\beta_f$ may lead to an algebraic proof. If the beta conjecture is false, showing that $\beta_f=0$ may be the easiest way to verify that one has a counterexample.

However, even if a proof of, or counterexample to,  the beta conjecture is difficult, there are other questions which are interesting and, perhaps, more approachable.

\bigskip

\begin{ques} Is  the beta conjecture true if $f$ is quasi-homogeneous?
\end{ques}

\medskip

\begin{ques} Is  the beta conjecture true if $\Sigma f$ is contained in a smooth surface? That is, after an analytic change of coordinates, is the beta conjecture true if $\Sigma f$ is contained in a $2$-plane?
\end{ques}

\medskip

It seems to be difficult to produce hypersurfaces with a critical locus which is $1$-dimensional, singular, irreducible, and with a small $\beta_f$. The case where $\beta_f=0$ is what is conjectured not to be possible. The case where $\beta_f=1$ is not possible by \thmref{thm:beta2}. Our question is:

\begin{ques} Is it possible for $\beta_f=2$ or $3$, if $f$ has a  critical locus which is $1$-dimensional, singular, and irreducible?
\end{ques} 

Related to the above question, we ask:

\begin{ques} Is there a relationship between $\beta_f$ and the Milnor number $\mu_\0(\Sigma f)$ of the curve $\Sigma f$, using the Milnor number of Buchweitz and Greuel in \cite{buchgreuel}? Is there, perhaps, some simple relationship, like something of the form $\beta_f\geq 2\mu_\0(\Sigma f)$?
\end{ques}

\medskip

\bibliographystyle{plain}
\bibliography{Masseybib}
\end{document}